\documentclass[11pt]{amsart}%
\usepackage{amsfonts}
\usepackage{amssymb}
\usepackage{amsmath}
\usepackage{graphicx}%
\usepackage{color}
\usepackage{amscd}
\usepackage[all,cmtip]{xy}
\newtheorem{theorem}{Theorem}[section]
\newtheorem{lemma}{Lemma}[section]
\newtheorem{proposition}{Proposition}[section]

\theoremstyle{definition}

\newtheorem{example}[theorem]{Example}

\theoremstyle{remark}
\newtheorem{remark}[theorem]{Remark}

\numberwithin{equation}{section}

\newcommand{\abs}[1]{\lvert#1\rvert}
\newcommand{\R}{\mathbb{R}}
\newcommand{\C}{\mathbb{C}}
\newcommand{\Z}{\mathbb{Z}}
\newcommand{\N}{\mathbb{N}}
\newcommand{\ov}{\overline}
\newcommand{\begeq}{\begin{equation}}
\newcommand{\stopeq}{\end{equation}}
\newcommand{\ep}{\epsilon}
\newcommand{\ci}{\mathbb{S}^1}
\newcommand{\dis}{\displaystyle}
\newcommand{\pa}{\partial}
\newcommand{\ei}[1]{\mathrm{e}^{#1}}
\newcommand{\dd}[2]{\dis\frac{\pa #1}{\pa #2}}
\newcommand{\Om}{\Omega}

\newcommand{\ta}{\theta}
\newcommand{\loj}{{\L}ojasiewicz}
\newcommand{\Si}{\Sigma}
\newcommand{\begar}{\begin{array}}
\newcommand{\stopar}{\end{array}}
\newcommand{\ke}{\mathcal{K}_{(x,y)}(\xi,\eta)}
\newcommand{\norm}[1]{\left\lVert#1\right\rVert}

\begin{document}
\title[Hypocomplex]
{A {\L}ojasiewicz Inequality in Hypocomplex Structures
of $\R^2$.}

\author{Abdelhamid Meziani}
\address{\small Department of Mathematics, Florida International University, Miami, FL, 33199, USA}
\email{meziani@fiu.edu}



\begin{abstract}
For a real analytic complex vector field $L$, in an open set of $\R^2$, with local first
integrals that are open maps, we attach a number
$\mu\ge 1$ (obtained through \loj\ inequalities) and show that the equation
$Lu=f$ has bounded solution when $f\in L^p$ with $p>1+\mu$.
We also establish a similarity principle between the bounded solutions of
the equation $Lu=Au+B\ov{u}$ (with $A,B\in L^p$) and holomorphic functions.
\end{abstract}
\maketitle

\section{Introduction}
This paper deals with the solvability of real analytic complex vector fields defined
in an open set $\mathcal{O}\subset\R^2$.
We assume that the vector field $L$ has local first integrals that are open maps everywhere.
This implies the existence of a global first integral $Z$ in $\mathcal{O}$ and allows us to
take $L$ as $L=Z_x\pa_y-Z_y\pa_x$. We associate to $L$ a number $\mu\ge 1$ (in fact
$\mu >1$ unless $L$ is elliptic everywhere).
 This number is derived from \loj\ inequalities for first integrals.
 We use a generalized Cauchy operator
 \[
 T_Zf(x,y)=\frac{-1}{\pi}\int_\Om f(\xi,\eta)\ke \, d\xi d\eta =
 \frac{-1}{\pi}\int_\Om \frac{f(\xi,\eta)\, d\xi d\eta}{Z(\xi,\eta)-Z(x,y)}
 \]
 to prove that for every $f\in L^p(\Om)$, with $p>1+\mu$, we have $u=T_Zf\, \in L^\infty(\Om)$
 and satisfies the equation $Lu=f$. We also prove a similarity between $L^\infty$-solution
 of the equation $Lu=Au+B\ov{u}$, where $A,B\in L^p(\Om)$, and the holomorphic functions $H$ defined
 in $Z(\Om)$. More precisely the solution $u$ can be written as $u=H(Z)\ei{s}$ with
 $H$ holomorphic in $Z(\Om)$ and $s\in L^\infty(\Om)$.

 The questions addressed in this paper are related to those contained in  papers:
 \cite{Beg},  \cite{Ber-Mez}, \cite{BCH}, \cite{BHS},  \cite{Bers},   \cite{CDM1}, \cite{CDM2},  \cite{CaMez},
 \cite{HZ}, \cite{Mez1}, \cite{Mez2}, \cite{Mez3}, \cite{Mez4}, \cite{Mez5},
 \cite{MAB}, \cite{MezZ}, \cite{Tre}
and others.

The organization of this paper is as follows. Section 2 deals with the necessary
background and terminology. In section 3 we derive normal forms for first integrals
that will be used in section 4 to introduce a \loj\  number $\mu$ for the vector field.
In section 5, we show that for $\dis 1<q<1+\frac{1}{\mu}$, the generalized Cauchy kernel
$\ke$ satisfies $\norm{\mathcal{K}_{(x,y)}}_{L^q(\Om)}\le C$ for all $(x,y)\in\mathcal{O}$
where $C$ is a positive constant depending on $\Om$, $q$ and $\mu$.
In section 6, we construct bounded solution of the equation $Lu=f$
for $f\in L^p(\Om)$, with $p>1+\mu$, and establish a similarity principle.

\section{Preliminaries}
In this section we set the necessary background and terminology used
in the study of Hypoanalytic structures (see \cite{BCH} and \cite{Tre}).
Let $\mathcal{O}\subset\R^2$ be open and connected and let
\begeq
L=A(x,y)\dd{}{x}+B(x,y)\dd{}{y}
\stopeq
with $A,B\,\in\, C^\varpi(\mathcal{O};\C)$ ($A,\, B$ real analytic and $\C$-valued in $\mathcal{O}$).
We assume that $L$ has no singular points: $|A|+|B|>0$ everywhere in $\mathcal{O}$.
The conjugate of $\ov{L}$ is the vector field $\ov{A}\pa_x+\ov{B}\pa_y$, where $\ov{A}$ and $\ov{B}$ are
the complex conjugates of $A$ and $B$. The vector field $L$ is elliptic at all points where
$L$ and $\ov{L}$ are independent. At each elliptic point, $L$ is locally conjugate to the CR operator $\dis\dd{}{\ov{z}}$.
Denote by $\Si$ the set of non elliptic points. Thus $\Si$ is the real analytic variety given by
\begeq
\Si =\left\{ (x,y)\in \mathcal{O}:\ \textrm{Im}(A(x,y)\ov{B(x,y)})=0\right\}\, .
\stopeq
A point $p\in\Si$ is of finite type if there exists a vector field $X$ in the Lie algebra generated by
$L$ and $\ov{L}$ such that $X$ and $L$ are independent at $p$. This means that there exists a Lie bracket
\[
X=\left[X_1,\left[X_2,\,\cdots\, \left[X_{n-1},X_{n}\right]\right]\right]
\]
with $X_i=L$ or $X_i=\ov{L}$, for $i=1,\cdots, n$, such that $X$ and $L$ are independent at p.
The smallest such $n$ is the type of $L$ at $p$. If no such $X$ exists, $L$ is said to be of
infinite type at $p$.

Since $L$ is real analytic and non singular, then it is locally integrable. That is, for every
$p\in\mathcal{O}$, there exists a $C^\varpi$-function $Z_p$ defined in an open neighborhood  $U_p$ of
$p$ such that $LZ_p(x,y)=0$ and $dZ_p(x,y)\ne 0$ for all $(x,y)\in U_p$ (the existence of such $Z_p$ in
the real analytic category is a direct consequence of the local straightening of the complexified
vector field $\hat{L}$ of $L$ in $\C^2$). It should be mentioned that in the $C^\infty$-category
the local integrability does not always hold. The simplest vector field for which the local integrability
fails is given by a $C^\infty$ perturbation at $0$ of the Mizohata vector field
$\pa_y-iy\pa_x$ (see \cite{Tre}).

From now on, we will assume that $L$ is of finite type at each point on $\Si$ and that for each
$p\in\Si$, the local first integral $Z_p$ is a homeomorphism from
$U_p$ to $Z(U_p)$. In fact the openness of $Z_p$ is equivalent to saying that the finite type of
$L$ at $p$ is an odd number. This is also equivalent to $L$ of finite type satisfies the Nirenber-Treves
condition($\mathcal{P}$) (see \cite{BCH} or \cite{Tre}). Such an $L$ for which the local first integrals are
homeomorphisms is called hypocomplex. A fundamental property of a hypocomplex vector field $L$ is the following:
If $u\in L^\infty(V)$ and $Lu=0$, where $V\subset\mathcal{O}$ is open,  then for $p\in V$ and every
first integral $Z_p:U_p\longrightarrow \C$, there exists a holomorphic function $h_p$ defined on
$Z_p(U_p\cap V)$ such that $u=h_p\circ Z_p$ in $U_p\cap V$.
This implies that the collection of charts $(U_p,Z_p)_{p\in\mathcal{O}}$ define a structure of a
Riemann surface on the open set $\mathcal{O}$.
It follows from the uniformization of planar Riemann surfaces (see \cite{Spr} or \cite{Sim} for a shorter
version of the proof) that $L$ has a global first integral. That is, there exists an injective
 $C^\varpi$-function $Z:\mathcal{O}\, \longrightarrow\, \C$ such that
 $LZ=0$ and $dZ\ne 0$ everywhere in $\mathcal{O}$.

\begin{remark}
It should be pointed out that in the paper \cite{BergCH} the authors construct a global first
integral for a vector field by assuming the ambient domain to be simply connected and using the
Reimann Mapping Theorem. In this paper, we don't assume simple connectedness
of the domain. But only that the generated Riemann surface is planar. Such Riemann surface have
a global parametrization (\cite{Spr}, \cite{Sim}).
\end{remark}

Without loss of generality, we will assume that
 \begeq\label{VecL}
 L=Z_x(x,y)\dd{}{y}-Z_y(x,y)\dd{}{x}
 \stopeq
 The characteristic set $\Si=\left\{\textrm{Im}(Z_x\ov{Z_y})=0\right\}$ is an analytic
 variety of dimension $\le 1$. We can decompose $\Si$ as $\Si=\Si^0\cup \Si^1$, where
 $\Si^0$ (if not empty) is an analytic variety of dimension 0 (so a collection of isolated points),
 and $\Si^1$ (if not empty) is a variety of dimension 1. The variety $\Si^1$ can be stratified as
 $\Si^1=\Si^1_R\cup S$, where $\Si^1_R$ is the regular part (union of 1-dimensional manifolds)
 and $S$ is the set of singular points (union of isolated points that are boundary points of
 components of $\Si^1_R$). Thus
 \begeq\label{Sigma}
\Si =\Si^0\cup\Si^1_R\cup S\, .
\stopeq
Moreover, for each connected component $\Gamma$ of $\Si^1_R$, the vector field $L$ has constant
finite type $k(\Gamma)$ at each point $p\in\Gamma$.  Indeed, for $p\in \Gamma$, there exists
local coordinates $(x,y)$ in an open neighborhood $U$ of $p$ such that $\Si\cap U =\{ y^l=0\}$ (for some $l\in \Z^+$.
The type of the vector
field thus constant on $\Gamma\cap U$ and this constant type propagates along $\Gamma$ until the curves ends at a singular point
or reaches the boundary of the domain. Note that the decomposition of $\Si$ is algebraic so that it could be geometrically
a one-dimensional manifold with the singular set $S\ne \emptyset$ (this is the case near 0 in the  example below).

 \begin{example}
 Let
 \[
 P(x,y)=y^2(x^2+y^2)(y-1+x^2)^6(x^2+(y+1)^2)(x^2+(y-3)^2-1)^2\, .
 \]
 Consider the function
 \[
 Z(x,y)=x+iQ(x,y)=x+i\int_0^yP(x,\tau)d\tau\, .
 \]
 Note since $P\ge 0$, then for each fixed $x$, the function $Q(x,.):\R\longrightarrow\R$
 is a homeomorphism. Thus $Z:\R^2\longrightarrow\C$ is a homeomorphism. The characteristic
 set of the hypocomplex structure defined by $Z$ is given by
 $\Si=P^{-1}(0)$. The decomposition of $\Si$ is as follows:
 \begin{itemize}
 \item $\ \Si^0=\{ (0,-1)\}$;
 \item $\ S=\{ (-1,0), (0,0), (1,0)\}$
 \item $\ \Si^1_R= l_1\cup l_2\cup l_3\cup l_4\cup \gamma^1\cup\gamma_2\cup \gamma_3\cup C$, where
 $l_1,\cdots ,l_4$ are the connected components of the $(\R\times\{0\})\backslash S$,
 $\gamma_1,\cdots ,\gamma_3$ are the arc of the connected components $\{y=1-x^2\} \backslash S$,
 and $C$ is the circle $x^2+(y-3)^2=1$.
 \end{itemize}
 \begin{figure}[ht]
 \centering
\scalebox{0.12} {\includegraphics{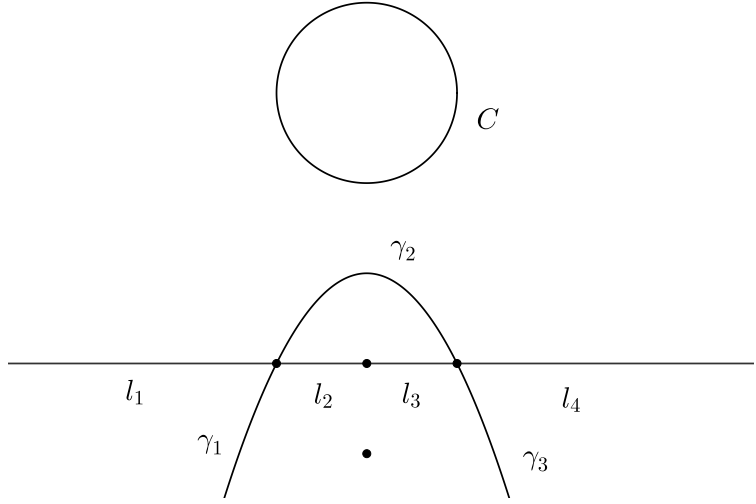}}
\caption{Characteristic set $\Si$ }
\end{figure}
The structure is of type 3 along each segment $l_1,\ l_2,\  l_3,\  l_4$;  of type 7 along the arcs of
parabola $\gamma_1,\ \gamma_2 ,\ \gamma_3$; and of type 5 along the circle $C$. At the isolated point $(0,-1)$,
it is of type 3. At the singular points, it is of type 9 at $(-1,0)$ and $(1,0)$ and of type 5 at
$(0,0)$.
 \end{example}

\section{Normal forms of first integrals}

In this section we derive normal forms for first integrals in neighborhoods of characteristic points
of a hypocomplex structure given by (\ref{VecL}).
Throughout, for $\ep$, $\delta$, and $\rho$ positive numbers, we will denote by
$D_\rho$ the disc with center $0$ and radius $\rho$ and by $R_{\ep,\rho}$ the rectangle in $\R^2$
given by
\[
R_{\ep,\rho}=\{(s,t)\in\R^2:\ -1-\ep<s<1+\ep,\ \ -\delta<t<\delta\}\, .
\]

The following proposition can be derived from the general theory of hypoanalytic structures (see \cite{BCH}
or \cite{Tre})

\begin{proposition}\label{NormalForm1}
Let $p\in\Si$. Then there exists a $C^\varpi$-diffeomorphism
\[
\Phi_p:\, D_\rho\,\longrightarrow\, U^p_\rho=\Phi_p( D_\rho)\subset \mathcal{O};\quad \Phi_p(0)=p
\]
and a function $\psi(s,t)\in C^\varpi(D_\rho; \R)$ satisfying $\psi(0)=0$ and
$\psi(s,t)\ge 0$ for all $(s,t)\in D_0(\rho)$ such that the function
\begeq\label{Normal1}
Z_p(s,t)=s+i\phi(s,t)=s+i\int_0^t\psi(s,\tau)d\tau
\stopeq
satisfies $d\Phi_p^\ast Z\wedge dZ_p =0$ in $D_\rho$. That is $Z_p$ is a first integral
of the vector field $(\Phi_p^{-1})_\ast L$.
\end{proposition}

\begin{remark}\label{TypeOrder}
The type of $L$ at the point $p$ can be computed through the normal form $Z_p$ and it is the order
at $t=0$ the function $\phi(0,t)$. If $p\in\Si^0$ (isolated point in $\Si$),
then $Z_p$ is a diffeomorphism in the
deleted disc $D_\rho\backslash 0$. This implies that $\psi(s,t)>0$ for $(s,t)\ne 0$ and that
$\phi$ can be written as
$\phi(s,t)=t\phi_1(s,t)$ with $\phi_1(s,t)>0$ for $(s,t)\ne 0$.

For $p\in S$,  it follows from $\psi\ge 0$ that there exists $r \ge 1$ (odd) such that
$\phi (s,t)=t^r\phi_1(s,t)$ with $\phi_1>0$ in $D_\rho\backslash 0$.
By the Weierstrass Preparation Theorem, we can write $\phi_1$ in a neighborhood of 0, as
\[
\phi_1(s,t)=M(s,t)\left(t^{2l}+\sum_{j=1}^{2l}A_j(s)t^{2l-j}\right)
\]
where $M$ and the $A_j$'s are germs of analytic functions, such that $M$ is a unit ($M(0)\ne 0$),
$A_j(0)=0$ for all $j=1,\cdots 2l$ and the order of $A_{2l}$ is even. The type of $L$ at $p$
is $k=r+2l$.
\end{remark}

Now we construct a normal form along characteristic curves of $\Si^1_R$.

\begin{proposition}\label{NormalForm2}
Let $\Gamma$ be a connected component of $\Si^1_R$ and let $\Gamma_{p,q}\subset\Gamma$ be
an arc connecting two distinct points $p,q\in\Gamma$. Let $k-1$ be the type of
$L$ along $\Gamma$.Then there exists  a $C^\varpi$-diffeomorphism
\[
\Phi_{p,q}:\, R_{\ep,\delta}\,\longrightarrow\, U^{p,q}_{\ep,\delta}=
\Phi_{p,q}( R_{\ep,\delta})\subset \mathcal{O};\quad \Phi_{p,q}(-1,0)=p,\ \ \Phi_{p,q}(1,0)=q
\]
such that the function
\begeq\label{Normal2}
Z_{p,q}(s,t)=s+it^k
\stopeq
satisfies $d\Phi_{p,q}^\ast Z\wedge dZ_{p,q} =0$ in $R_{\ep,\delta}$. That is $Z_{p,q}$ is a first integral
of the vector field $(\Phi_{p,q}^{-1})_\ast L$.
\end{proposition}

\begin{proof}
First note that since $L$ is not a (complex) multiple of a  tangent vector at any point
on $\Gamma$ (otherwise that point of tangency would be a point of infinite type), and since
$dZ\ne 0$ everywhere, then the restriction of $Z$ to $\Gamma$ is a $C^\varpi$-differomorphism.
Hence $\Gamma'=Z(\Gamma) $ is a 1-dimensional submanifold in $\C$. Note also that
$Z(\Si)$ is a subanalytic set in $\C$.
Let $\Gamma_{p,q}'=Z(\Gamma_{p,q}) $ and let $U$ be an open neighborhood of $\Gamma_{p,q}'$
such that $U\backslash\Gamma'$ consists of two simply connected components $U^+$ and $U^-$
sharing a common boundary along $\Gamma'$. Let
$H^+:\ov{U^+}\longrightarrow H^+(\ov{U^+})\subset\C$ be a conformal map such that
\[
H^+(\Gamma'\cap \ov{U^+})\,\subset\, \R+i0,\quad H^+(Z(p))=-1,\ \ \textrm{and}\ \
H^+(Z(q))=1\, .
\]
Since $\Gamma'$ is real analytic, it follows from the Schwarz reflection principle that
$H^+$ can be extended  across $\Gamma'$ and into $U^-$ as a holomorphic $H$ function
defined in a full neighborhood $V$ of $\Gamma_{p,q}'$ and $H'(z)\ne 0$ for every $z\in V$.

Let $\gamma :\R\, \longrightarrow\, \Gamma$ a $C^\varpi$-parametrization such that
$\gamma(-1)=p$, $\ \gamma(1)=q$ and $\gamma([-1,\ 1])=\Gamma_{p,q}$. Let
$\eta :[-1-\ep,\ 1+\ep]\,\longrightarrow\, \R^2$ be $C^\varpi$ and such that
$\eta(\sigma)$ is a unit normal vector to $\Gamma$ at the point $\gamma(\sigma)$.
Define $\Phi_{p,q}$ on $R_{\ep,\delta}$ as
\[
\Phi_{p,q}(\sigma,\tau)=\gamma(\sigma)+\tau\eta(\sigma)\, .
\]
Then for $\ep,\ \delta$ small enough $\Phi_{p,q}$ is a diffeomorphism onto
$ U^{p,q}_{\ep,\delta}=\Phi_{p,q}( R_{\ep,\delta})$ and
\[
\Phi_{p,q}\left((-1-\ep,\ 1+\ep)\times\{0\}\right)=U^{p,q}_{\ep,\delta}\cap\Si \,\subset\, \Gamma\, .
\]
That is, the characteristic set of $(\Phi_{p,q}^{-1})_\ast L$ in $R_{\ep,\delta}$ consists of
the $\sigma$-axis and each point has finite type $k-1$.
The function
\[Z_{p,q}(\sigma,\tau)=H\circ Z\circ\Phi_{p,q}(\sigma,\tau)
\]
 is a first integral of $(\Phi_{p,q}^{-1})_\ast L$.
 We have $Z_{p,q}(\sigma ,0)\in\R$ and $\dis\dd{Z_{p,q}}{\sigma}(\sigma,0)\ne 0$
  for every $\sigma\in (-1-\ep,\ \ 1+\ep)$. We can therefore (after shrinking $\ep$ and $\delta$
  if necessary), take $s=\textrm{Re}( Z_{p,q}(\sigma ,\tau)$  as a new variable. With respect to the variables
  $(s,\tau)$ the first integral becomes $Z_{p,q}(s,\tau)=s+i\beta(s,\tau)$. Since the characteristic set $\tau=0$
  has type $k-1$, then $\beta$ has the form $\beta(s,\tau)=\tau^k\beta_1(s,\tau)$ with $\beta_1(s,0)\ne 0$ for all $s$.
  Finally the with respect to the coordinates $(s,t)$ with $\dis t=\tau\, \sqrt[k]{\beta_1(s,\tau)} $,
  the first integral $Z_{p,q}$ has the desired form (\ref{Normal2})
\end{proof}

By using similar arguments as those used in the previous proposition, we can prove the following.

\begin{proposition}\label{NormalForm3}
Let $\Gamma\subset\Si^1_R$ be a closed curve. Assume that $L$ is of type $k-1$ along $\Gamma$.
Then there exists  a $C^\varpi$-diffeomorphism
\[
\Phi_\Gamma:\, \ci\times (-\delta,\ \delta)\,\longrightarrow\, U^\Gamma_{\delta}=
\Phi_\Gamma(\ci\times (-\delta,\ \delta))\subset \mathcal{O};\quad \Phi_\Gamma\left(\ci\times\{0\}\right)=\Gamma
\]
such that the function
\begeq\label{Normal3}
Z_\Gamma =\ei{t^k+i\ta}\,
\stopeq
where $\ta$ is the angular variable on $\ci$,
satisfies $d\Phi_\Gamma^\ast Z\wedge dZ_\Gamma =0$ in $\ci\times (-\delta,\ \delta)$.
That is $Z_\Gamma$ is a first integral
of the vector field $(\Phi_\Gamma^{-1})_\ast L$.
\end{proposition}

We will refer to $\left(U^p_\rho,\Phi_p, Z_p\right)$ ,
$\left(U^{p,q}_{\ep,\delta},\Phi_{p,q}, Z_{p,q}\right)$, and
$\left(U^\Gamma_\delta,\Phi_\Gamma, Z_\Gamma\right)$
given in Propositions \ref{NormalForm1}, \ref{NormalForm2}, and \ref{NormalForm3},
as charts of normal forms for the vector field $L$ or the first integral $Z$.

\section{\loj\  numbers for the vector field $L$}

The \loj\ inequality (see \cite{Bi-Mi}, \cite{Loj}  for example) found its way
as an important tool in many areas including Algebraic Geometry, Differential Equations, and Optimization.
The version of the inequality that we will use is the following. Given an open set $U\subset\R^n$,
and $\R$-valued real analytic functions $f$ and $g$ defined on $U$. If $f^{-1}(0)\subset g^{-1}(0)$, then
for every compact set $K\subset\subset U$, there exist positive constants $C$ and $\mu$ such that
\begeq\label{LojIne}
\abs{f(x)}\ge C\abs{g(x)}^\mu \qquad \forall x\in K\, .
\stopeq
%
%In fact we will use the parametric version of this inequality
%(see \cite{Fek}, Section 4 in \cite{Rab}, Theorem 2.2 in \cite{Tou}):
%Let $U\subset \R^n$, $V\subset \R^m$ be open sets, $F(x,t)$ and $G(x,t)$ be $\R$-valued real analytic
%functions in $U\times V$. Suppose that for every $t\in V$, $\, F^{-1}(0,t)\subset G^{-1}(0,t)$.
%Then for every compact sets $K\subset\subset U$ and $K'\subset\subset V$, there exist positive
%constants $C=C(K,K')$ and $\mu=\mu(K,K')$ such that
%\begeq\label{LojIne}
%\abs{F(x,t)}\ge C\abs{G(x,t)}^\mu\qquad \forall (x,t)\in K\times K'\, .
%\stopeq

We are going to associate a positive number to a hypocomplex structure defined by a real analytic vector field $L$.
This number is defined through the \loj\ exponent appearing in inequality (\ref{LojIne}) and is linked to the type
of the vector field.
 We start with a normal form chart
$\left(U^{p,q}_{\ep,\delta},\Phi_{p,q}, Z_{p,q}\right)$ about the characteristic arc
$\Gamma_{p,q}$ along which $L$ is of type $k$, with normal first integral $Z_{p,q}(s,t)=s+it^k$.
It follows from the inequality $\abs{t^k-b^k}\ge \abs{t-b}^k$ for all $t,b\in\R$, that
\begeq\label{LojIneq1}
\abs{Z_{p,q}(s,t)-Z_{p,q}(a,b)}^2\,\ge\, {(s-a)^2+(t-b)^{2k}}\quad \forall (s,t),\, (a,b)\in R_{\ep,\delta}
\, .
\stopeq
We refer to $k$ as the \loj\ number of $L$ in the chart $\left(U^{p,q}_{\ep,\delta},\Phi_{p,q}, Z_{p,q}\right)$.

For $p\in\Si^0\cup S$, let $\left(U^p_\rho,\Phi_p, Z_p\right)$ be the normal form chart about the point $p$
with $Z_p(s,t)=s+i\phi(s,t)$ and $\phi(s,t)=t^r\phi_1(s,t)$ and $\phi_1(s,t)>0$ for $(s,t)\ne(0,0)$ such that
$r=1$ when $p\in\Si_0$, and $r\ge 1$ (odd number) when $p\in S$ (see Remark 3.1). We can assume without loss of generality
that $\phi$ is defined in an open set containing the closed disc $\ov{D_\rho}$.
Consider the real analytic functions $F$ and $G$ given in $\R^4$  by
\[\left\{\begar{ll}
F\left((s,t),(a,b)\right) & =\dis\frac{(s-a)^2}{2}+\left[\phi(s,t)-\phi(a,b)\right]^2\\ \\
G\left((s,t),(a,b)\right) & =\dis (t-b)^2
\stopar\right.\]
We have $G^{-1}(0)=\{ t=b\}$ and since the function $Z_p$ is a homeomorphism,
then $F^{-1}(0)=\{ s=a,\ t=b\}$. Clearly $F^{-1}(0)\subset G^{-1}(0)$.
Therefore, it follows from (\ref{LojIne}) that there exist $C=C(\rho) >0$  and $\mu=\mu(\rho)>0$ such that
\begeq\label{LojIneq11}
\dis\frac{(s-a)^2}{2}+\left[\phi(s,t)-\phi(a,b)\right]^2\ge C\abs{t-b}^{2\mu}\quad \forall (s,t),\ (a,b)\in \ov{D_\rho(0)}.
\stopeq
Define the \loj\ exponent at the point $p$ as
\begeq\label{LojExp}
\mu_p=\inf\left\{ \mu>0;\ \exists C>0,\,\rho>0,\ \text{(\ref{LojIneq11})\ {holds in}}\ D_\rho^2
\right\}
\stopeq

\begin{lemma}\label{Type=Loj}
Let $k$ and $\mu_p$ be the type and the \loj\ exponent of $L$ at the point $p\in\Si$. Then
$k=\mu_p$.
\end{lemma}

\begin{proof}
First recall (Remark \ref{TypeOrder}) that the type $k$ of $L$ at the point $p$ is the order at $t=0$ of the function
$\phi(0,t)$.
Suppose $p\in \Si^0\cup S$. By taking $s=a=0$ and $b=0$ in inequality (\ref{LojIneq11}) we get
$\abs{\phi(0,t)}\ge C\abs{t}^\mu$ for $\abs{t}<\rho$. Therefore $\mu\ge k$ and so $\mu_p\ge k$.
To prove equality, suppose, by contradiction that $\mu_p >k$. Let $\mu_0$ such that $k<\mu_0<\mu_p$.
It follows from the definition of $\mu_p$ that there exist sequence of points
$\{ A_n=(a_n,b_n)\}_{n\in\Z^+}$ and $\{ X_n=(s_n,t_n)\}_{n\in\Z^+}$
with
$A_n,\ X_n, \in D_{1/n}$  and
\begeq\label{LojExpoProof}
\frac{(s_n-a_n)^2}{2}+\left[\phi(X_n)-\phi(A_n)\right]^2\, <\, \abs{t_n-b_n}^{2\mu_0}\quad
\forall n\in\Z^+
\stopeq
Set $s_n-a_n=\alpha_n$ and $t_n-b_n=\beta_n$ and expand the function $\phi$ into its Taylor series
about the point $A_n$ as
\[
\phi(X_n)-\phi(A_n)=\sum_{i+j\ge 1,\ i\ge 1}\frac{\pa^{i+j}\phi(A_n)}{\pa s^i\pa t^j}\frac{\alpha_n^i\beta_n^j}{i!\, j!}+
\sum_{j=0}^\infty \frac{\pa^{j}\phi(A_n)}{\pa t^j}\frac{\beta_n^j}{ j!}\, .
\]
Note that it follows from (\ref{LojExpoProof}) and $k<\mu_0$ that $\dis\lim_{n\to\infty}\frac{\alpha_n}{\beta_n^k}=0$.
It also follows from the uniform convergence of Taylor series that
\begeq\label{Limzero}
\lim_{n\to\infty}
\sum_{i+j\ge 1,\ i\ge 1}\frac{\pa^{i+j}\phi(A_n)}{\pa s^i\pa t^j}\frac{\alpha_n^i\beta_n^j}{\beta_n^k i!\, j!}=0.
\stopeq
 To see why (\ref{Limzero}) holds, we can use the Cauchy integral formula to the complexified function
$\hat{\phi}(\hat{s},\hat{t})$ in a closed bidisc $\ov{D(\rho)}^2\subset\C^2$ with $\rho >0$ small to get estimates of the derivatives
\[
\left|\frac{\pa^{i+j}\hat{\phi}(\hat{s},\hat{t})}{\pa \hat{s}^i\pa \hat{t}^j}\right|
\le 4\pi^2M\frac{i! \, j!}{(\rho /2)^i(\rho/2)^j}\quad \forall (\hat{s},\hat{t})\in D(\rho/2)^2\, ,
\]
where $\dis M=\max_{|\hat{s}|=|\hat{t}|=\rho}\abs{\hat{\phi}(\hat{s},\hat{t})}$. In particular, for $n$ large enough, we have
\[
\left|\frac{\pa^{i+j}\phi(A_n)}{\pa s^i\pa t^j}\right|\le 4\pi^2M\frac{i! \, j!}{(\rho /2)^i(\rho/2)^j}\, .
\]
Consequently, for $n$ large enough so that $\alpha_n<\rho/2$ and $\beta_n<\rho/2$, we have
\[
\left|\sum_{i+j\ge 1,\ i\ge 1}\frac{\pa^{i+j}\phi(A_n)}{\pa s^i\pa t^j}\frac{\alpha_n^i\beta_n^j}{\beta_n^k i!\, j!}\right|
\le \frac{4\pi^2 M\rho^2}{(\rho-2\alpha_n)(\rho-2\beta_n)}\, \frac{2\alpha_n}{\beta_n^k}
\]
and the limit in (\ref{Limzero}) follows.
As a consequence, we obtain after dividing (\ref{LojExpoProof}) by $\beta_n^{2k}$ and taking the limit
that
\[
\lim_{n\to\infty}\left(\sum_{j=0}^\infty \frac{\pa^{j}\phi(A_n)}{\pa t^j}\frac{\beta_n^j}{j!\,\beta_n^k }\right) =0.
\]
This is a contradiction since $\dis \frac{\pa^{j}\phi(0)}{\pa t^j}=0$ for $j<k$ and $\dis \frac{\pa^{k}\phi(0)}{\pa t^k}\ne 0$.
\end{proof}
Consequently
\begeq\label{LojIneq2}
\abs{Z_{p}(s,t)-Z_{p}(a,b)}^2\,\ge\, M\left({(s-a)^2+\abs{t-b}^{2\mu_0}}\right)\quad \forall (s,t),\, (a,b)\in \ov{D_\rho}
\, ,
\stopeq
with $M=\dis\min\left(\frac{1}{2},C\right)$.
We refer to $\mu_0=k$ as the \loj\ number of $L$ in the chart $\left(U^p_\rho,\Phi_p, Z_p\right)$.

Let $\Omega$ be open, relatively compact, and $\Om\subset\subset\mathcal{O}$. Then
$\Si_\Om=\Si\cap\ov{\Om}$ can be covered by a finite number of open sets
$U_1,\cdots ,U_N$ such that each $U_j$ is a normal form chart for the vector field $L$.
Let $\mu_j$ be the \loj\ number of $L$ in $U_j$. We define the \loj\ number of
$L$ in $\Om$ as
\begeq\label{Lojnumber}
\mu=\mu(L,\Om)=\max\{\mu_1,\cdots ,\mu_N\}\, .
\stopeq

\begin{remark}
The definition (\ref{Lojnumber}) simply means that the \loj\ number of $L$ in $\ov{\Om}$
is obtained by taking the largest type of $L$ in $\ov{\Om}$, which is a finite number because the
vector field is real analytic and $\Om$ is bounded. This number is also invariant under
change of coordinates and does not depend on the open sets  $U_1,\cdots ,U_N$.
\end{remark}

\section{Estimate for a generalized Cauchy kernel}

We establish an integral estimate that will be crucial in the solvability of the vector field $L$.
We start with a simple lemma.

\begin{lemma}\label{SimpleLemma}
Let $0<\tau<1$ and $1<q<1+\tau$. There exists a positive constant $C=C(\tau,q)$ such that for $\rho >0$ we have
\begeq\label{SimpleEstimate}
I=\int_{D_\rho}\frac{ds\, dt}{\abs{s+i|t|^{1/\tau}}^q}\le C\max(\rho^{1+\tau-q},\rho^{(1+\tau-q)/\tau})
\stopeq
\end{lemma}

\begin{proof}
For $t\ne0$ consider the new variable $\eta =\dis\frac{t}{|t|}\abs{t}^{1/\tau}$ so that
$dt=\tau \abs{\eta}^{\tau-1}$ and for $\rho'= \max(\rho,\rho^{1/\tau})$ the integral can be estimated
with the use of polar coordinates as
\[
I\le \int_{D_{\rho'}}\frac{\tau ds\, d\eta}{\abs{\eta}^{1-\eta}\,\abs{s+i\eta}^q}
=\left(\int_0^{2\pi}\frac{\tau d\ta}{\abs{\sin\ta}^{1-\tau}}\right)\int_0^{\rho'}r^{\tau-q}dr
\]
and (\ref{SimpleEstimate}) follows.
\end{proof}

Let $\mathcal{O}\subset\R^2$ be open and connected and $Z:\mathcal{O}\,\longrightarrow\, Z(\mathcal{O})\subset\C$
be a real analytic homeomorphism and let $L=\dis Z_y\pa_x-Z_x\pa_y$.
For $(x,y)$ and $(\xi,\eta)$ distinct points in $\mathcal{O}$, consider the function $\ke$ given
by
\begeq\label{kernel}
\ke =\frac{1}{Z(\xi,\eta)-Z(x,y)}\, .
\stopeq
We have the following proposition
\begin{proposition}\label{KEstimate}
Let $\Om\subset\mathcal{O}$ be open and relatively compact and let
$\mu=\mu(L,\Om)$ be the \loj\ number of $L$ on $\Om$.
Then for $\dis 1<q<1+\frac{1}{\mu}$, we have $\mathcal{K}_{(x,y)}\in L^q(\Om)$.
Moreover, there exists a positive constant $C=C(\Om, q,\mu)$ such that
\begeq\label{KEstimate-q}
\norm{\mathcal{K}_{(x,y)}}_{L^q(\Om)}\le C\qquad \forall (x,y)\in \mathcal{O}\, .
\stopeq
\end{proposition}

The remainder of this section deals with the proof of the proposition. We start by proving (\ref{KEstimate-q})
in the case of a normal form

\begin{lemma}\label{EstimateNormal}
Let $Z_0(s,t)=s+i\phi(s,t)$ with $\phi\in C^\varpi(D_\rho;\R)$, $\, \phi(s,t)=t^r\phi_1(s,t)$ with
$r\in\Z^+$ and $\phi_1(s,t)>0$ for $(s,t)\ne 0$, and such that $Z_0: D_\rho\longrightarrow Z_0(D_\rho)$ is a homeomorphism.
Let $\mu_0$ be the associated \loj\ number of the induced hypocomplex structure and let
$\dis 1<q<1+\frac{1}{\mu_0}$. Then there exists $C=C(\rho, \mu_0,q)$ such that
\begeq\label{Normal-q-estimate1}
I_w=\int_{D_{\rho/2}}\frac{ds\,dt}{\abs{Z_0(s,t)-w}^q}\,\le\, C\qquad \forall w\in\C\, .
\stopeq
\end{lemma}

\begin{proof}
If $w\notin Z_0(D_\rho)$, then there exists $\delta >0$ such that
\[
\abs{Z_0(s,t)-w}\ge \delta\qquad\forall (s,t)\in D_{\rho/2}\, .
\]
Therefore
\[
I_w\le \frac{\pi\rho^2}{4\delta^q}\, .
\]
For $w\in Z_0(D_\rho)$, let $(a,b)\in D_\rho$ such that $w=Z_0(a,b)$.
In this case the \loj\ inequality implies that
\[
\abs{Z_0(s,t)-w}^2=\abs{Z_0(s,t)-Z_0(a,b)}^2\ge C_\rho \left[(s-a)^2+(t-b)^{2\mu_0}\right]
\]
and, after using $D_{\rho/2}\subset D((a,b), 2\rho)$ and applying Lemma (\ref{SimpleLemma}), we get
\[\begar{ll}
I_w & \dis\le  \frac{1}{C_\rho}\int_{D_{\rho/2}}\frac{ds\, dt}{\left[(s-a)^2+(t-b)^{2\mu_0}\right]^{q/2}}\\ \\
& \dis\le \frac{1}{C_\rho}\int_{D((a,b),2\rho)}\frac{ds\, dt}{\left[(s-a)^2+(t-b)^{2\mu_0}\right]^{q/2}}\\ \\
&\dis\le \frac{1}{C_\rho}\int_{D_{2\rho}}\frac{ds\, dt}{\abs{s+i|t|^{\mu_0}}^{q}}\le C(\rho,\mu_0,q)
\stopar
\]
\end{proof}

Similar arguments show that in a normal form chart
 $\left(U^{p,q}_{\ep,\delta},\Phi_{p,q}, Z_{p,q}\right)$ about the characteristic arc
$\Gamma_{p,q}$ along which $L$ is of type $k-1$, with normal first integral $Z_{p,q}(s,t)=s+it^k$,
we have
\begeq\label{Normal-q-estimate2}
I_w=\int_{R_{\ep/2,\delta/2}}\frac{ds\,dt}{\abs{Z_{p,q}(s,t)-w}^q}\,\le\, C\qquad \forall w\in\C\, .
\stopeq

Now we continue with the proof of the proposition.
Let $\Si_\Om=\Si\cap\ov{\Om}$. Then $\Si_\Om$ can be covered by a finite number of normal form charts
$U_1,\,\cdots\, ,U_N$. Hence for $j=1,\cdots ,N$, there exists a diffeomorphism
$\Phi_j: V_j\, \longrightarrow\, U_j$ with $V_j=D_\rho$ or $V_j=R_{\ep,\delta}$. Denote by
$Z_j$ the normal form of the first integral in $V_j$.
Let $V_j'=D_{\rho/2}$ or $V_j'=R_{\ep/2,\delta/2}$, respectively, and let $U'_j=\Phi_j(V'_j)$.
We can assume that $\dis\Si_\Om\subset \bigcup_{j=1}^NU'_j$.
After shrinking $\rho$, $\ep$ and $\delta$, if necessary, we can find conformal maps $H_j$ defined on
open neighborhoods of $\ov{Z_j(V_j)}$ such that
\begeq\label{ZtoNormal}
Z(x,y)=H_j\left(Z_j\left(\Phi_j^{-1}(x,y)\right)\right)\ \ \text{in}\ U_j,\qquad j=1,\cdots ,N\, .
\stopeq
For $j=1,\cdots ,N$, consider the positive constants
\[
c_j=\inf_{z\in Z_j(V_j)}\abs{H_j'(z)}\quad\text{and}\quad M_j=\sup_{(s,t)\in V_j}\abs{D\Phi_j(s,t)}
\]
and, for $(x,y)\in \mathcal{O}$, the integral
\begeq\label{IntegralInUj}
I_j=\int_{U'_j}\abs{\ke}^q d\xi\,d\eta
\stopeq
If $(x,y)\notin U_j$, then there exists $\delta_j >0$ such that $\abs{\ke}\ge \delta$,
$\forall (\xi,\eta)\in U'_j$ and in this case
\begeq\label{EstimateIj1}
I_j\le \frac{\text{Area}(U'_j)}{\delta_j^q}\qquad\forall (x,y)\in \mathcal{O}\backslash U_j\, .
\stopeq
If $(x,y)\in U_j$, then there exist $(a,b)\in V_j$ such that
$(x,y)=\Phi_j(a,b)$. Also  $(\xi,\eta)\in U_j$ can be written as $(\xi,\eta)=\Phi_j(s,t)$ with
$(s,t)\in V_j$. Hence,
\[
\abs{\ke} =\frac{1}{\abs{H_j\left(Z_j(s,t)\right)-H_j\left(Z_j(a,b)\right)}}
\le \,\frac{1}{c_j\, \abs{Z_j(s,t)-Z_j(a,b)}}\, .
\]
It follows from this inequality, (\ref{Normal-q-estimate1}) and (\ref{Normal-q-estimate2})that
\begeq\label{EstimateIj2}\begar{ll}
I_j & \dis \le\frac{1}{c_j^q}\int_{V_j'}\frac{\abs{D(\Phi_j(s,t))}\, ds\, dt}{\abs{Z_j(s,t)-Z_j(a,b)}^q} \\ \\
& \dis\le \frac{M_j}{c_j^q}\int_{V_j'}\frac{ ds\, dt}{\abs{Z_j(s,t)-Z_j(a,b)}^q}\le C_j\qquad
\forall (x,y)\in U_j
\stopar\stopeq
for some constant $C_j$ depending on $U_j$, $q$, and the \loj\ number $\mu_j$ in $U_j$.
It follows then from (\ref{EstimateIj1}) and (\ref{EstimateIj2}) that for every $j=1,\cdots N$ there
exists $B_j=B_j(U_j,q,\mu_j) $ such that
\begeq\label{EstimateIJ}
I_j\le B_j\qquad \forall (x,y)\in\mathcal{O}\, .
\stopeq
Let $\dis \Om_0=\Om\backslash \left(\bigcup_{j=1}^NU_j'\right)$.  The structure induced by $Z$ is
elliptic in an open neighborhood of $\Om_0$ and therefore $\ke$ is conjugate to the classical
Cauchy kernel $\dis\frac{1}{\zeta -z}$ and it follows that there exists a positive constant
$B_0=B_0(\Om_0,q)$ such that
\begeq\label{EstimateI0}
I_0=\int_{\Om_0}\abs{\ke}^q\, d\xi\,d\eta \,\le\, B_0 \qquad\forall (x,y)\in\mathcal{O}.
\stopeq
Finally, for $\dis 1<q<1+\frac{1}{\mu}\le 1+\frac{1}{\mu_j}$, $j=1,\cdots, N$ we have
\[
\norm{\mathcal{K}_{(x,y)}}_{L^q(\Om)}\le \sum_{j=0}^NI_j\, \le\, \sum_{j=0}^NB_j\qquad
\forall (x,y)\in \mathcal{O} \ .
\]

\section{Generalized Cauchy  Operator}

A generalization of the Cauchy integral operator for  vector fields in two variables
appeared in papers \cite{Mez3}, \cite{Mez4}, \cite{Mez5} and then in \cite{CDM2} and \cite{CaMez}.
Here we use the generalized Cauchy operator for real analytic hypocomplex structures
given by a vector filed (\ref{VecL}) with a global first integral $Z$
on the open set $\mathcal{O}\subset\R^2$.
For $\Om\subset\mathcal{O}$ open and relatively compact, we consider the
operator
\label{OpeTZ}
\begeq\begar{ll}
T_Zu(x,y) & \dis =\frac{-1}{\pi }\int_\Om u(\xi,\eta)\,\ke \, d\xi d\eta\\ \\
& \dis =
\frac{-1}{\pi }\int_\Om \frac{u(\xi,\eta)}{Z(\xi,\eta)-Z(x,y)} \, d\xi d\eta
\stopar\stopeq
We have the following theorem

\begin{theorem}\label{TZinLInfinty}
Let $\mu=\mu(\Om,L)$ be the \loj\ number of $L$ on $\Om$ and let $p>1+\mu$.
Then there exists a positive constant $M=M(\Om,p,\mu)$ such that
\begeq\label{TZEstimate}
\abs{T_Zf(x,y)}\,\le\, M\,\norm{f}_{L^p(\Om)}\qquad
\forall f\in L^p(\Om),\ \ \forall (x,y)\in \mathcal{O}\, .
\stopeq
\end{theorem}

\begin{proof}
For $p>1+\mu$, the H\"{o}lder conjugate $q$ satisfies $\dis 1<q<1+\frac{1}{\mu}$.
The H\"{o}lder inequality gives
$\dis \abs{T_Zf(x,y)}\,\le\, \norm{{\mathcal{K}_{(x,y)}}}_{L^q(\Om)}\,\norm{f}_{L^p(\Om)}$
and (\ref{TZEstimate}) follows from Proposition (\ref{KEstimate}).
\end{proof}

\begin{proposition}\label{FinL1}
Let $f\in L^1(\Om)$, then $T_Z f\in L^q(\Om)$ for any $q$ with
$\dis 1<q<1+\frac{1}{\mu}$.
\end{proposition}

\begin{proof}
Let $\dis 1<q<1+\frac{1}{\mu}$ and $p >1+\mu$ be the H\"{o}lder conjugate of $q$.
Let $g$ be an arbitrary function in $L^p(\Om)$. It follows from Theorem \ref{TZinLInfinty} that
$g_1(x,y)\in L^\infty(\mathcal{O})$ where $g_1$ is given by
\[
g_1(x,y)=\int_\Om \abs{g(\xi,\eta)}\, \abs{\ke}\, d\xi\,d\eta\, .
\]
Therefore $fg_1\in L^1(\Om)$. Now we use Fubini's Theorem together with
$\mathcal{K}_{(x,y)}(\xi,\eta)=-\mathcal{K}_{(\xi,\eta)}(x,y)$ to obtain
\[
\dis\int_\Om\abs{f(x,y)}g_1(x,y)\, dx dy  =
\int_\Om\abs{g(\xi,\eta)}
\underbrace{\left(\int_\Om\abs{f(x,y)}\,\abs{\mathcal{K}_{(\xi,\eta)}(x,y)}dx\, dy\right)}_{f_1(\xi,\eta)}d\xi\,d\eta
\]
This means that $gf_1\in L^1(\Om)$. Since $g\in L^p(\Om)$ is arbitrary, then the reverse of the H\"{o}lder
inequality implies that $f_1\in L^q(\Om)$. We have $\abs{T_Zf}\,\le f_1$, and so $T_Zf\in L^q(\Om)$.
\end{proof}

Arguments similar to those used to prove Proposition 12 in \cite{CDM2}  can be applied to establish the
following

\begin{proposition}\label{CauchyIntegralFormula}
Let $w\in C(\ov{\Om})\cap C^1(\Om)$. Then for every $(x,y)\in \Om$, we have
\begeq\label{IntegralFormula}
2\pi i w(x,y)=\int_{\partial\Om}\!\!\!w(\xi,\eta)\, \ke\, dZ(\xi,\eta)\, +
\int_{\Om}\!\!\! Lw(\xi,\eta)\, \ke\,d\zeta\wedge d\ov{\zeta},
\stopeq
where $\zeta =\xi+i\eta$.
\end{proposition}

\begin{theorem}\label{TZSolution}
If $f\in L^1(\Om)$, then $L(T_Zf)=f$. In particular if
$f\in L^p(\Om)$ with $p>1+\mu$, where $\mu$ is the \loj\ number of $L$ in
$\Om$, then $T_Zf$ is an $L^\infty$-solution of the equation
$Lu=f$.
\end{theorem}

\begin{proof}
We know from Proposition (\ref{FinL1}) that for $f\in L^1(\Om)$, $T_Zf\in L^q(\Om)$ for
$\dis 1<q<1+\frac{1}{\mu}$. Note that the transpose of the operator $L$ is $^tL=-L$.
Let $\phi\in C^\infty_0(\Om)$ be a test function. Then by using (\ref{IntegralFormula}) for
$\phi$ and taking into account that $\phi=0$ on $\pa\Om$, we obtain
\[\begar{ll}
\left<L(T_Zf),\phi \right> &\dis = -\left<T_Zf,L\phi\right>\\ \\
& \dis =
\int_\Om\left[\frac{1}{\pi }\int_\Om f(\xi,\eta)\ke d\xi\, d\eta\right]L\phi(x,y)\, dx\, dy\\ \\
& \dis =
\int_\Om f(\xi,\eta)\left[\frac{1}{2\pi i}\int_\Om L\phi(x,y)\mathcal{K}_{(\xi,\eta)}(x,y) dz\wedge d\ov{z}\right]\, d\xi\, d\eta\\ \\
& \dis =\left<f,\phi\right>
\stopar\]
The second part of the statement of the Theorem follows from Theorem \ref{TZinLInfinty}.
\end{proof}

Now we establish a similarity principle between the solutions of the equation
$Lu=Au+B\ov{u}$ and holomorphic functions in $Z(\Om)$. More precisely, we have
the following

\begin{theorem}\label{Similarity}
Let $A,\, B\, \in L^p(\Om)$ with $p>1+\mu$, where $\mu$ is the \loj\ number of
$L$ in $\Om$. If $u\in L^\infty(\Om)$ satisfies the equation
\begeq\label{LuAuBubar}
Lu=Au+B\ov{u}\, ,
\stopeq
then there exists a holomorphic function $H$ in $Z(\Om)$ and a function
$s\in L^\infty(\Om)$ such that
\begeq\label{UHeS}
u(x,y)=H(Z(x,y))\, \ei{s(x,y)}\, .
\stopeq
\end{theorem}

\begin{proof}
Let $u\in L^\infty(\Om)$ be a solution of (\ref{LuAuBubar}).
Since $L$ is elliptic in $\mathcal{O}\backslash\Si$, then if $u=0$ on a set with
an accumulation in $\Om$, then $u=0$ everywhere in $\Om\backslash\Si$.
For $u$ not identically zero,  consider the function
$\chi$ given by $\dis \chi =\frac{\ov{u}}{u} $ at the points where $u\ne 0$
and $\chi=0$ when $u=0$.  Then $\abs{\chi}\le 1$ and  $A+B\chi \in L^p(\Om)$.
It follows from Theorem \ref{TZSolution} that the function
$s=T_Z(A+B\chi)\in L^\infty(\Om)$ and satisfies
\begeq\label{EqforS}
uLs =Au+B\ov{u}\, .
\stopeq
Consider the function $v\in L^\infty(\Om)$ given by $v=u\ei{-s}$. Then
it follows from (\ref{EqforS}) that $Lv=0$ in $\Om$. Since $L$ is hypocomplex,
then $v$ factors as $v=H\circ Z$ for some holomorphic function $H$ defined
in $Z(\Om)$ and (\ref{UHeS}) follows.
\end{proof}

\begin{remark}
The results of this paper are stated for real analytic structures but they can be extended
for any structure that satisfies the \loj\ inequality in neighborhoods of points of
the characteristic set $\Si$. In particular, if
$Z:\, \mathcal{O}\, \longrightarrow\, Z(\mathcal{O})\subset\C$ is $C^{1+\alpha}$-homeomorphism
($0<\alpha <1$) and if $Z$ is real analytic in a tubular neighborhood of $\Si$, then a \loj\
number $\mu$  can be attached to $L$ on any relatively compact open set $\Om\subset\mathcal{O}$
and Theorems \ref{TZinLInfinty}, \ref{TZSolution}, and \ref{Similarity} hold.
In fact the real analyticity of $Z$ near $\Si$ can also be weakened to only assume that $Z$ is
in Denjoy-Carleman classes, since functions in these classes satisfy the \loj\ inequality
(see \cite{Bi-Mi2} or \cite{Vol}).

Hypotheses on $f$ and $A,\, B$ in equations $Lu=f$ and $Lu=Au+B\ov{u}$ can also be weakened
to assume that $f,\, A,\,  B$ are in $L^{p_0}(\Om)$ for some $p_0>2$ and in
$L^{p_i}(U_i)$ with $p_i>1+\mu_i$ in each normal form chart $U_i$ with \loj\ number $\mu_i$.
\end{remark}

\end{document}